\documentclass[11pt]{article}
\usepackage{amsmath, amssymb, amsthm}
\usepackage{verbatim}
\usepackage{multicol}
\usepackage{enumerate}
\usepackage{comment}
\usepackage[none]{hyphenat}
\usepackage{hyperref}
\hypersetup{
	colorlinks=true,
	linkcolor=blue,
	filecolor=magenta,      
	urlcolor=cyan,
	citecolor=blue
}
\usepackage{pgf}
\usepackage{tikz}
\usetikzlibrary{positioning,arrows,shapes,decorations.markings,decorations.pathreplacing,matrix,patterns}
\tikzstyle{vertex}=[circle,draw=black,fill=black,inner sep=0,minimum size=3pt,text=white,font=\footnotesize]
\usepackage{cleveref}

\date{}
\title{\vspace{-0.8cm}Improved Ramsey-type results for comparability graphs}
\author{
D\'aniel Kor\'andi \thanks{Institute of Mathematics, EPFL
, Lausanne, Switzerland. \, Emails: \{daniel.korandi,istvan.tomon\}@epfl.ch. Research supported in part by SNSF grants 200020-162884 and 200021-175977. }
\and
Istv\'an Tomon \footnotemark[1]
}

\theoremstyle{plain}
\newtheorem{theorem}{Theorem}

\newtheorem{claim}[theorem]{Claim}
\newtheorem{lemma}[theorem]{Lemma}
\newtheorem{conjecture}[theorem]{Conjecture}
\newtheorem{problem}[theorem]{Problem}
\newtheorem{prop}[theorem]{Proposition}
\Crefname{theorem}{Theorem}{Theorems}
\Crefname{definition}{Definition}{Definitions}
\Crefname{corollary}{Corollary}{Corollaries}
\Crefname{claim}{Claim}{Claims}
\Crefname{lemma}{Lemma}{Lemmas}
\Crefname{conjecture}{Conjecture}{Conjectures}
\Crefname{problem}{Problem}{Problems}
\Crefname{prop}{Proposition}{Propositions}

\theoremstyle{definition}

\DeclareMathOperator{\rank}{rk}
\DeclareMathOperator{\polylog}{polylog}

\newcommand{\eps}{\varepsilon}
\newcommand{\subs}{\subseteq}

\renewcommand{\Pr}{\mathbb{P}}

\newcommand{\calS}{\mathcal{S}}

\begin{document}

\maketitle
\sloppy

\begin{abstract}
Several discrete geometry problems are equivalent to estimating the size of the largest homogeneous sets in graphs that happen to be the union of few \emph{comparability graphs}. An important observation for such results is that if $G$ is an $n$-vertex graph that is the union of $r$ comparability (or more generally, perfect) graphs, then either $G$ or its complement contains a clique of size $n^{1/(r+1)}$.

This bound is known to be tight for $r=1$. The question whether it is optimal for $r\ge 2$ was studied by Dumitrescu and T\'oth. We prove that it is essentially best possible for $r=2$, as well: we introduce a probabilistic construction of two comparability graphs on $n$ vertices, whose union contains no clique or independent set of size $n^{1/3+o(1)}$.

Using similar ideas, we can also construct a graph $G$ that is the union of $r$ comparability graphs, and neither $G$, nor its complement contains a complete bipartite graph with parts of size $\frac{cn}{(\log n)^r}$. With this, we improve a result of Fox and Pach.
\end{abstract}

\section{Introduction}

An old problem of Larman, Matou\v{s}ek, Pach and T\"or\H{o}csik \cite{LMPT94,PT94} asks for the largest $m=m(n)$ such that among any $n$ convex sets in the plane, there are $m$ that are pairwise disjoint or $m$ that pairwise intersect. Considering the \emph{disjointness graph} $G$, whose vertices correspond to the convex sets and edges correspond to disjoint pairs, this problem asks the Ramsey-type question of estimating the largest clique or independent set in $G$. The best known lower bound, proved in \cite{LMPT94}, is based on the fact that every disjointness graph is the union of four comparability graphs.

$G$ is a \emph{comparability graph} if its edges correspond to comparable pairs in some partially ordered set on its vertex set $V(G)$. It is well-known that every comparability graph $G$ is perfect, i.e., for every induced subgraph $H$ of $G$, the chromatic number $\chi(H)$ is equal to the size of the largest clique $\omega(H)$. We will refer to cliques and independent sets as \emph{homogeneous sets}. 

As noted by Dumitrescu and T\'oth \cite{DT02}, unions of perfect graphs have strong Ramsey properties.

\begin{prop}
\label{basicthm}
If a graph $G$ is the union of $r$ perfect graphs, then $\omega(G)^{r}\ge \chi(G)$. In particular, such a $G$ contains a homogeneous set of size $n^{1/(r+1)}$.
\end{prop}

\begin{proof}
Let $G_{1},\dots,G_{r}$ be perfect graphs whose union is $G$. Let $\chi_{i}$ be a proper coloring of $G_{i}$ with $\omega(G_{i})$ colors. Then the coloring $\chi$ defined by $\chi(v)=(\chi_{1}(v),\dots,\chi_{r}(v))$ for $v\in V(G)$ is a proper coloring of $G$ with at most $\omega(G_{1})\cdots\omega(G_{r}) \le \omega(G)^{r}$ colors.

As $\alpha(G)\chi(G)\ge n$, we either have $\alpha(G)\ge n^{1/(r+1)}$ or $\chi(G)\ge n^{r/(r+1)}$. The latter implies $\omega(G)\ge n^{1/(r+1)}$.
\end{proof}

In \cite{LMPT94}, Larman, Matou\v{s}ek, Pach and T\"or\H{o}csik show that the disjointness graph of convex sets is the union of four comparability graphs, and use \Cref{basicthm} to deduce the existence of $n^{1/5}$ pairwise disjoint or pairwise intersecting sets among them. In the special case when each convex set is a half line, they note that the disjointness graph can actually be written as the union of two comparability graphs, so we even get $n^{1/3}$ such sets. These bounds have not been improved in the past 25 years, even though it is not even clear if \Cref{basicthm} is tight in general.

\medskip
The problem of estimating the largest size of homogeneous sets in the union of comparability graphs was raised by Dumitrescu and T\'oth \cite{DT02}. They defined $f_r(n)$ as the largest integer such that the union of any $r$ comparability graphs on $n$ vertices contains a homogeneous set of size $f_r(n)$. It is easy to see that $f_1(n)=\lceil\sqrt{n}\rceil$. More generally, Dumitrescu and T\'oth proved
\[ n^{1/(r+1)} \le f_r(n) \le n^{(1+\log_{2}r)/r} \]
by complementing \Cref{basicthm} with an appropriate blow-up construction of comparability graphs. 
For $r=2$, they found a somewhat better construction to establish $f_2(n)\le n^{0.4118}$. This upper bound was subsequently improved to $f_2(n)\le n^{0.3878}$ by Szab\'{o} \cite{S,DT02}. Our first result shows that \Cref{basicthm} is essentially sharp for $r=2$, i.e., $f_2(n)=n^{1/3+o(1)}$.

\begin{theorem}\label{mainthm}
Let $n$ be a positive integer. There is a graph $G$ on $n$ vertices that is the disjoint union of two comparability graphs, such that the largest homogeneous set of $G$ has size $n^{1/3}(\frac{\log n}{\log \log n})^{2/3}$.
\end{theorem}

The closely related problem of finding large complete or empty bipartite graphs in the union of comparability graphs was investigated by Fox and Pach \cite{FP09} and the second author \cite{T16}. We define a \emph{biclique} as a complete bipartite graph with parts of equal size. Let $p_r(n)$ denote the largest number such that for any graph $G$ on $n$ vertices that is the union of $r$ comparability graphs, either $G$ or its complement contains a biclique of size $p_r(n)$. Fox and Pach proved that
\[
n\cdot e^{-c_r (\log\log n)^{r}} \, < \, p_{r}(n) \, = \, O_{r}\left(n\frac{(\log\log n)^{r-1}}{(\log n)^{r}}\right).
\]
In $\cite{T16}$, these results are extended in a Tur\'{a}n type setting.

As our next result, we show that the $(\log\log n)^{r-1}$ factor can be removed in the upper bound. While this might not seem like a substantial improvement, we find it reasonable to believe that this new upper bound is sharp.
 
\begin{theorem}\label{thm:bipartite}
	For every positive integer $r$, there is a constant $c=c(r)>0$ such that the following holds. For every positive integer $n$, there is a graph $G$ on $n$ vertices that is the union of $r$ comparability graphs, and neither $G$ nor its complement contains a biclique of size $\frac{cn}{(\log n)^r}$.
\end{theorem}

We organize the paper as follows. We give the proof of \Cref{mainthm} in \Cref{sect:clique} and the proof of \Cref{thm:bipartite} in \Cref{sect:bipartite}. We discuss some further connections to geometry and open problems in \Cref{sect:open}. We systematically omit floor and ceiling signs whenever they are not crucial.

\section{Small cliques and independent sets}\label{sect:clique}

\begin{figure}
\begin{center}
	\begin{tikzpicture}[scale=1]

	\draw[fill=black, fill opacity=.1] 
	(0.75,6) -- (1,6.25) -- (2,5.25) -- (1.75,5) -- (0.75,6);
	\draw[fill=black, fill opacity=.1] 
	(0.75,4) -- (1,4.25) -- (2,3.25) -- (1.75,3) -- (0.75,4);
	\draw[fill=black, fill opacity=.1] 
	(0.75,2) -- (1,2.25) -- (2,1.25) -- (1.75,1) -- (0.75,2);
	
	\draw[blue, dashed, fill=blue, fill opacity=.1] 
	(6.75-0.75,6) -- (6.75-1,6.25) -- (6.75-2,5.25) -- (6.75-1.75,5) -- (6.75-0.75,6);
	\draw[blue, dashed, fill=blue, fill opacity=.1] 
	(6.75-0.75,4) -- (6.75-1,4.25) -- (6.75-2,3.25) -- (6.75-1.75,3) -- (6.75-0.75,4);
	\draw[blue, dashed, fill=blue, fill opacity=.1] 
	(6.75-0.75,2) -- (6.75-1,2.25) -- (6.75-2,1.25) -- (6.75-1.75,1) -- (6.75-0.75,2);

	\foreach \i in {0,...,3}
	{
		\foreach \j in {0,...,3}
		{
        \draw[fill=white] (2*\i+0.375,2*\j+0.625) circle (0.75);
			\foreach \k in {0,...,3}
			{
			
		\node[vertex] (A\i\j\k) at (2*\i+0.25*\k,2*\j+1-0.25*\k) {};
        	}
		}
	}	
  \node[] at (-1,6.75) {$V_{1,1}$};
  \node[] at (-1,0.25) {$V_{b,1}$};
  \node[] at (7.75,0.25) {$V_{b,b}$};
  \node[] at (7.75,6.75) {$V_{1,b}$};

  \node at (1.375,3.625) {$<_1$};
  \node[blue] at (5.375,3.625) {$<_2$};

   \foreach \i[evaluate=\i as \evali using int(\i+1)] in {0,...,2}
   {
    \foreach \j in {0,...,3}
    {
     \foreach \k in {0,...,3}
     {
     	\draw[blue, dashed] (2*\j+0.75-0.25*\k,2*\i+0.25*\k+0.25) -- (2*\j+0.75-0.25*\k,2*\i+2+0.25*\k+0.25) ;
     }	
    }
   }
   
\draw (0.75,0.25) -- (2.00,1.00) ;
\draw (0.50,0.50) -- (2.50,0.50) ;
\draw (0.25,0.75) -- (2.75,0.25) ;
\draw (0.00,1.00) -- (2.25,0.75) ;
\draw (0.75,2.25) -- (2.75,2.25) ;
\draw (0.50,2.50) -- (2.50,2.50) ;
\draw (0.25,2.75) -- (2.25,2.75) ;
\draw (0.00,3.00) -- (2.00,3.00) ;
\draw (0.75,4.25) -- (2.50,4.50) ;
\draw (0.50,4.50) -- (2.00,5.00) ;
\draw (0.25,4.75) -- (2.25,4.75) ;
\draw (0.00,5.00) -- (2.75,4.25) ;
\draw (0.75,6.25) -- (2.00,7.00) ;
\draw (0.50,6.50) -- (2.50,6.50) ;
\draw (0.25,6.75) -- (2.25,6.75) ;
\draw (0.00,7.00) -- (2.75,6.25) ;
\draw (2.75,0.25) -- (4.00,1.00) ;
\draw (2.50,0.50) -- (4.25,0.75) ;
\draw (2.25,0.75) -- (4.50,0.50) ;
\draw (2.00,1.00) -- (4.75,0.25) ;
\draw (2.75,2.25) -- (4.00,3.00) ;
\draw (2.50,2.50) -- (4.75,2.25) ;
\draw (2.25,2.75) -- (4.25,2.75) ;
\draw (2.00,3.00) -- (4.50,2.50) ;
\draw (2.75,4.25) -- (4.50,4.50) ;
\draw (2.50,4.50) -- (4.00,5.00) ;
\draw (2.25,4.75) -- (4.25,4.75) ;
\draw (2.00,5.00) -- (4.75,4.25) ;
\draw (2.75,6.25) -- (4.75,6.25) ;
\draw (2.50,6.50) -- (4.25,6.75) ;
\draw (2.25,6.75) -- (4.50,6.50) ;
\draw (2.00,7.00) -- (4.00,7.00) ;
\draw (4.75,0.25) -- (6.75,0.25) ;
\draw (4.50,0.50) -- (6.25,0.75) ;
\draw (4.25,0.75) -- (6.50,0.50) ;
\draw (4.00,1.00) -- (6.00,1.00) ;
\draw (4.75,2.25) -- (6.50,2.50) ;
\draw (4.50,2.50) -- (6.75,2.25) ;
\draw (4.25,2.75) -- (6.25,2.75) ;
\draw (4.00,3.00) -- (6.00,3.00) ;
\draw (4.75,4.25) -- (6.50,4.50) ;
\draw (4.50,4.50) -- (6.75,4.25) ;
\draw (4.25,4.75) -- (6.25,4.75) ;
\draw (4.00,5.00) -- (6.00,5.00) ;
\draw (4.75,6.25) -- (6.00,7.00) ;
\draw (4.50,6.50) -- (6.50,6.50) ;
\draw (4.25,6.75) -- (6.25,6.75) ;
\draw (4.00,7.00) -- (6.75,6.25) ;
   
\end{tikzpicture}
\caption{Our construction for \Cref{mainthm}. Horizontal edges are in $<_1$, vertical edges are in $<_2$. Most diagonal edges are omitted for clarity.}
\label{figure1}
\end{center}
\end{figure}
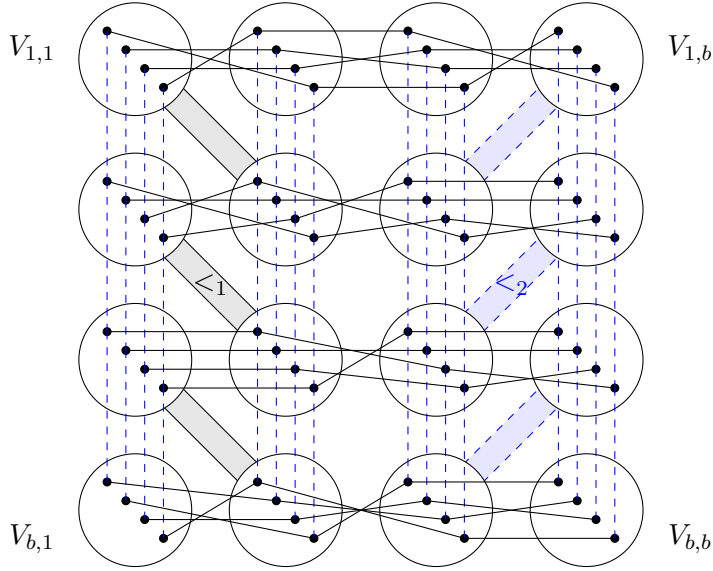

In this section, we prove \Cref{mainthm}. We start with describing our construction (see \Cref{figure1}).

Let $a= n^{1/3}(\frac{\log n}{\log \log n})^{2/3}$ and $b= n^{1/3}(\frac{\log\log n}{\log n})^{1/3}$. Then $n=ab^{2}$. Let $\{V_{i,j}\}_{(i,j)\in [b]^{2}}$ be a family of $b^{2}$ pairwise disjoint sets of size $a$. We will define two disjoint partial orders $<_{1}$ and $<_{2}$ on $V=\bigcup_{(i,j)\in [b]^{2}}V_{i,j}$, where the vertices in each ``row'' $R_{i}=\bigcup_{j=1}^{b}V_{i,j}$ will form $a$ disjoint chains in $<_{1}$, the vertices in each ``column'' $C_{j}=\bigcup_{i=1}^{b}V_{i,j}$ will form $a$ disjoint chains in $<_{2}$, and any two vertices not in the same row and column will be comparable in exactly one of the partial orders.

Actually, $<_1$ will be a random partial order, where each such chain starts at a random element of some $V_{i,1}$, and continues in $V_{i,2},\dots, V_{i,b}$ by selecting each successor uniformly at random. More precisely, take a uniformly random bijection $f_{i,j}:[a]\to V_{i,j}$ independently for every pair $(i,j)\in [b]^2$. Then every vertex in $V$ can be uniquely written as $f_{i,j}(k)$ for some $i,j\in[b], k\in[a]$. Now for two vertices $v=f_{i,j}(k)$ and $w=f_{i',j'}(k')$, we define $v<_{1}w$ iff
\begin{itemize}
	\item  $i<i'$ and $j<j'$, or
	\item  $i=i'$, $j<j'$ and $k=k'$.
\end{itemize}
For fixed $i$ and $k$, we refer to the set of vertices $\{f_{i,j}(k) : j\in[b]\}$ as the $k$'th chain in $R_i$.

The definition of $<_{2}$ is similar, except we do not need it to be random. So fix any bijections $g_{i,j}:[a]\to V_{i,j}$. For $v=g_{i,j}(\ell)$ and $w=g_{i',j'}(\ell')$, we define $v<_{2}w$ iff
\begin{itemize}
	\item  $i<i'$ and $j>j'$, or
	\item  $i<i'$, $j=j'$ and $\ell=\ell'$.
\end{itemize} 
Once again, we call the set $\{g_{i,j}(\ell) : i\in[b]\}$ the $\ell$'th chain in $C_j$.

\medskip
Let $G$ be the union of the comparability graphs of $<_{1}$ and $<_{2}$. \Cref{prop:indepbound,prop:cliquebound} below will show that the largest homogeneous sets in $G$ have size $a= n^{1/3}(\frac{\log n}{\log \log n})^{2/3}$ whp.\footnote{We say that an event holds \emph{with high probability} (whp) if it holds with probability $1-o(1)$ as $n\to \infty$.} In fact, $G$ contains both cliques and independent sets of size roughly $a$.

\begin{prop} \label{prop:indepbound}
$\alpha(G)=a$. 
\end{prop}

\begin{proof}
$V_{i,j}$ is independent, so $\alpha(G)\ge a$. Let $I\subs V$ be any independent set of $G$. If $v\in V_{i,j}$ and $w\in V_{i',j'}$ such that $vw$ is not an edge of $G$, then we must have $i=i'$ or $j=j'$. Hence, there is an $l\in [b]$ such that $I\subs R_{l}$ or $I\subs C_{l}$. Without loss of generality, $I\subs R_{l}$. But $I$ intersects each of the $a$ chains in $R_{l}$ in at most one element, so $|I|\le a$ and hence $\alpha(G)\le a$.
\end{proof}

The following observation is a convenient characterization of maximal cliques.

\begin{claim} \label{lem:maxclique}
Every maximal clique $S$ in $G$ can be written as $S=(\cup_{i\in [b]} X_i)\cap (\cup_{j\in [b]} Y_j)$, where $X_i$ is a chain in $R_i$ and $Y_j$ is a chain in $C_j$ for every $i,j\in[b]$.
\end{claim}

\begin{proof}
Let $S$ be a clique in $G$. Then $S$ intersects at most one chain in $R_i$ for each $i\in[b]$, because two vertices of $R_i$ in different chains are incomparable in both $<_1$ and $<_2$. Similarly, $S$ intersects at most one chain in $C_j$ for each $j\in[b]$, so $S\subs (\cup_{i\in [b]} X_i)\cap (\cup_{j\in [b]} Y_j)$ for some chains $X_i$ and $Y_j$. But $(\cup_{i\in [b]} X_i)\cap (\cup_{j\in [b]} Y_j)$ is always a clique, so we must have equality if $S$ is maximal.
\end{proof}

\begin{prop} \label{prop:cliquebound}
With high probability, $a/20\le \omega(G)\le a$.
\end{prop}

\begin{proof}
The crucial observation we need here is that each chain in $R_i$ intersects each chain in $C_j$ with probability $1/a$, i.e., for every $i,j\in[b]$ and $k,\ell\in [a]$, 
\[ \Pr[ f_{i,j}(k)=g_{i,j}(\ell) ] = 1/a, \]
and that these events are independent for different pairs $(i,j)$.

Now to prove the upper bound, fix $k_1,\dots,k_b,\ell_1,\dots,\ell_b\in[a]$, and let $X_i$ and $Y_j$ be the $k_i$'th and $\ell_j$'th chains in $R_i$ and $C_j$, respectively, for every $i,j\in[b]$. The probability that $\bigcup_{i=1}^b X_i$ intersects $\bigcup_{j=1}^b Y_j$ in at least $a$ elements is then at most $\binom{b^2}{a}\cdot (\frac{1}{a})^a \le (\frac{eb^2}{a^2})^a$.
Using \Cref{lem:maxclique} and a union bound over all choices of $k_i$ and $\ell_j$, we get that the probability that $G$ contains a clique of size $a$ is at most
\[ a^{2b} \cdot \left(\frac{eb^2}{a^2}\right)^a = \exp \Big( 2b\log a - 2a\log (a/b) + a \Big).\]
Here $a=b\frac{\log n}{\log\log n}<n^{1/2}$, so $\log(a/b) =\log\log n-\log\log\log n>\frac{2}{3}\log\log n$ and $\log a <\frac{1}{2}\log n$. Plugging this in, we get that the probability of a clique larger than $a$ is at most
\[ \exp \left( b\log n - \frac{4}{3}a \log\log n +a \right) = \exp \left( b\log n - \frac{4}{3}b \log n +b\frac{\log n}{\log\log n} \right) = o(1).\]

\medskip
For the lower bound, $X_i$ will be the first chain in $R_i$ for every $i\in [b]$, and then for each $j\in[b]$, we select $Y_j$ to be the chain in $C_j$ that has the largest intersection with $\cup_{i\in[b]} X_i$. We claim that whp this intersection has size at least $\frac{\log a}{3\log\log n}$ for at least half of the $j\in[b]$.

Indeed, finding the maximum intersection size is equivalent to the famous ``balls into bins'' problem, where $b$ balls are thrown into $a$ bins independently and uniformly at random, and we want to find the maximum number of balls in the same bin.

\begin{theorem}[Raab--Steger, \cite{RS98}] \label{thm:ballsbins}
If $ \frac{a}{\polylog a} \le b \ll a\log a$, and $b$ balls are thrown into $a$ bins independently and uniformly at random, then whp the maximum number of balls in the same bin is
\[ \frac{\log a}{\log \frac{a\log a}{b}} (1+o(1)). \]
\end{theorem}

In our case, $f_{i,j}(1)$ is an element of a uniformly random chain in $C_j$, independently for each $i$, so by \Cref{thm:ballsbins}, and using the fact that $ \frac{a\log a}{b} \le \log^2 n$, we get that $N_j=|Y_j \cap (\cup_{i\in [b]}X_i)|\ge \frac{\log a}{3\log\log n}$ whp, for each fixed $j\in[b]$. This means that the expected number of indices $j$ such that $N_j<\frac{\log a}{3\log\log n}$ is $o(b)$, so by Markov's inequality, $N_j\ge \frac{\log a}{3\log\log n}$ for at least $b/2$ different $j\in [b]$ whp. Hence,
\[
\sum_{j\in[b]} N_j = \Big|\big(\bigcup_{i\in [b]} X_i\big)\cap \big(\bigcup_{j\in [b]} Y_j\big)\Big| \ge \frac{b}{2}\cdot \frac{\log a}{3\log\log n} \ge \frac{b}{18} \cdot \frac{\log n}{\log\log n} = \frac{a}{18}
\]
whp, so $(\cup_{i\in [b]} X_i)\cap (\cup_{j\in [b]} Y_j)$ is a clique we were looking for.
\end{proof}

\section{Small bicliques}\label{sect:bipartite}

In this section, we prove \Cref{thm:bipartite}. We can actually prove a slightly stronger result.

\begin{theorem}\label{epsilon}
	For every $\eps>0$ and positive integer $r$, there is a constant $c$ such that the following holds. For $n$ sufficiently large, there is a graph $G$ on $n$ vertices that is the union of $r$ comparability graphs, $G$ has at most $n^{1+\eps}$ edges, and the complement of $G$ does not contain a biclique of size $\frac{cn}{(\log n)^r}$.
\end{theorem}

Clearly, if we choose $\eps<1$, then the bound on the number of edges ensures that there is no biclique of size $\frac{cn}{(\log n)^r}$ in $G$, either, establishing \Cref{thm:bipartite}.

Before we explain the construction, we need to make some technical definitions. For a graph $H$ and $U\subs V(H)$, let $N_H[U]=U\cup \bigcup_{v\in U} N_H(v)$ denote the closed neighborhood of $U$ in $H$. An $(n,d,\lambda)$-expander graph, is a $d$-regular graph $H$ on $n$ vertices such that for every $U\subs V$ satisfying $|U|\le |V|/2$, we have $|N_H[U]|\ge (1+\lambda)|U|$. A well-known result of Bollob\'as \cite{B88} shows that a random 3-regular graph on $n$ vertices is whp an $(n,3,\delta)$-expander for some $\delta>0$. For explicit constructions see, e.g., \cite{M94}.

Let $H^k$ denote the usual graph power of $H$, that is, $V(H^k)=V(H)$, and $v$ and $w$ are adjacent in $H^k$ if $H$ contains a path of length at most $k$ between $v$ and $w$. Here, we allow loops, so every vertex is joined to itself in $H^{k}$ for $k\ge 0$. We will use the following easy property of expander graphs. 

\begin{claim}\label{expander} 
	Let $H$ be an $(n,d,\lambda)$-expander and let $k\ge 1$. If $X,Y\subs V(H)$ such that there is no edge between $X$ and $Y$ in $H^{k}$, then $|X||Y|\leq n^{2}(1+\lambda)^{-k}$.
\end{claim}

\begin{proof}
Let $X_i=N_{H^i}[X]$ and $Y_i=N_{H^i}[Y]$ for $i=0,1,\dots,k$. As there are no edges between $X$ and $Y$, we know that $X_i$ and $Y_{k-i}$ are disjoint for every $i$. Let $\ell\in \{0,1,\dots,k\}$ be largest so that $|X_{\ell}|\le n/2$.
    
If $\ell=k$, then we can use the definition of expanders to show by induction on $j$ that $|X_{k-j}|\le n(1+\lambda)^{-j}/2$. In particular, $|X|\le n(1+\lambda)^{-k}/2<n(1+\lambda)^{-k}$, and hence $|X||Y|\le n^2(1+\lambda)^{-k}$.

On the other hand, if $\ell<k$, then $|X_{\ell+1}|>n/2$ and hence $|Y_{k-\ell-1}|\le n/2$. A similar inductive argument then gives $|X|\le n(1+\lambda)^{-\ell}/2$, as well as $|Y|\le  n(1+\lambda)^{-(k-\ell-1)}/2$. As $1+\lambda\le 2$, this gives
\[ |X||Y| \le n^2(1+\lambda)^{-k+1}/4 \le n^2(1+\lambda)^{-k}. \]
\end{proof}

Fox \cite{F06} used expander graphs to show that \Cref{epsilon} holds in the case $r=1$. Our proof for general $r$ combines his construction with ideas used in \Cref{mainthm}. In fact, the $r=2$ case of the construction shown below is quite similar to the one we described in the previous section.

\medskip
We first define $r$ auxiliary partial orders $\prec_1,\dots,\prec_r$ on $\mathbb{R}^{r}$ as follows. For $\alpha = (\alpha_1,\alpha_2,\dots,\alpha_r)$, $\beta = (\beta_1,\beta_2,\dots,\beta_r)$, and $1\le s\le r$, let $\alpha\preceq_{s}\beta$ if 
\[ \beta_{s}-\alpha_{s}=\max_{i\in [r]}|\beta_{i}-\alpha_{i}|. \]
It is easy to check that the relations $\prec_{1},\dots,\prec_{r}$ are indeed partial orders, and that for any $\alpha\ne\beta$, there is an $s$ such that $\alpha\prec_s\beta$ or $\beta\prec_s\alpha$ (although this $s$ might not be unique). For $A,B\subs \mathbb{R}^{r}$, let $A\preceq_{s}B$ if $\alpha\preceq_{s} \beta$ holds for every $\alpha\in A$, $\beta\in B$. Also, let $|\beta-\alpha|_{\infty}=\max_{i}|\beta_{i}-\alpha_{i}|$ be the usual $\ell_{\infty}$-norm.

\medskip

We are now ready to describe our construction. Set $b=\frac{\eps\log n}{\log 9}$ and $a=n/b^r$, and fix an $(a,3,\delta)$-expander $H$ on $a$ vertices. 

Let $\{V_{\alpha}\}_{\alpha\in [b]^r}$ be a family of $b^r$ disjoint sets of size $a$, and let  $V=\bigcup_{\alpha\in [b]^r}V_{\alpha}$ be the union of them. If $v\in V_{\alpha}$, we define the \emph{rank} of $v$ as $\rank(v)=\alpha$. We identify the elements of each $V_{\alpha}$ with the vertices of $H$. More precisely, let $h:V\to V(H)$ be a function that is a bijection when restricted to $V_{\alpha}$ for every $\alpha\in [b]^r$.

The partial orders $<_1,\dots,<_r$ are defined on $V$ as follows. For $v\in V_{\alpha}$ and $w\in V_{\beta}$ and every $1\le s\le r$, we let $v<_s w$ if $\alpha\prec_s\beta$, and $h(v)$ and $h(w)$ are joined by an edge in $H^{|\alpha-\beta|_{\infty}}$. Let us first check that these are indeed partial orders. 
 
\begin{claim}
 The relation $<_s$ is a partial order for every $1\le s\le r$.
\end{claim}

\begin{proof}
The only thing we need to check is that $<_s$ is transitive. So pick three vertices $u\in V_{\alpha}$, $v\in V_{\beta}$ and $w\in V_{\gamma}$ such that $u<_s v$ and $v<_s w$. Then $\alpha \prec_s \beta \prec_s \gamma$. Also, there is a path $P$ in $G$ of length at most $|\alpha-\beta|_{\infty}$ from $h(u)$ to $h(v)$, and another  path $P'$ of length at most $|\beta-\gamma|_{\infty}$ from $h(v)$ to $h(w)$. But then the union of $P$ and $P'$ contains a path of length at most 
\[ |\alpha-\beta|_{\infty}+|\beta-\gamma|_{\infty} = (\beta_s-\alpha_s) + (\gamma_s-\beta_s) = \gamma_s-\alpha_s =|\alpha-\gamma|_{\infty} \]
(using  $\alpha \prec_s \beta \prec_s \gamma$) from $h(u)$ to $h(w)$, thus indeed, $u<_s w$.
 \end{proof}

Let $G$ be the union of the comparability graphs of $<_s$ over all $1\le s\le r$.
The next lemma bounds the maximum degree, and hence the number of edges of $G$.

\begin{lemma} \label{nedges}
 The maximum degree of $G$ is at most $b^r 3^{b}$.
\end{lemma}

\begin{proof}
Let $v\in V_{\alpha}$. As $H$ is $3$-regular, the maximum degree of $H^k$ is at most $1+3+\dots+3^{k}< 3^{k+1}$. This means that the number of neighbors of $v$ in any $V_{\beta}$ is at most $3^{|\alpha-\beta|_\infty+1}\le 3^b$. Summing this for every $\beta\in[b]^r$, the number of neighbors of $v$ in $G$ is at most $b^{r}3^{b}$.
\end{proof}

It is more difficult to show that the complement of $G$ does not contain a large biclique. 

\medskip
For two sets $X,Y\subs V$, we write $X\preceq_s Y$ if $\rank(x)\preceq_s \rank(y)$ for every $x\in X$ and $y\in Y$. First, we show that for every $X,Y\subs V$, we can find relatively large sets $X'\subs X$, $Y'\subs Y$ and an $s\in [r]$ such that $X'\preceq_{s} Y'$ or $Y'\preceq_{s} X'$.

We start with a geometric claim. We say that a hyperplane $H$ separates two set $A$ and $B$, if $A$ is in one \emph{closed} halfspace determined by $H$, and $B$ is in the other.

\begin{claim}\label{cl:separation}
Let $A$ and $B$ be finite multisets of $\mathbb{R}^{d}$ such that $|A|=|B|=m$, and let $H$ be a hyperplane. There are subsets $A'\subs A$, $B'\subs B$ of size $|A'|=|B'|\ge m/2$ that are separated by a translate of $H$.
\end{claim}

\begin{proof}
Let $v$ be a vector orthogonal to $H$, and for $t\in\mathbb{R}$, let $H(t)=\{x\in \mathbb{R}^{d}:\langle v,x\rangle=t\}$, $H^{-}(t)=\{x\in \mathbb{R}^{d}:\langle v,x\rangle\leq t\}$ and $H^{+}(t)=\{x\in \mathbb{R}^{d}:\langle v,x\rangle\geq t\}$, where $\langle.,.\rangle$ denotes the usual dot product on $\mathbb{R}^{d}$. Let $t$ be minimum such that either $|H^{-}(t)\cap A|\geq m/2$ or $|H^{-}(t)\cap B|\geq m/2$. Without loss of generality, suppose that $|H^{-}(t)\cap A|\geq m/2$. Then $|H^{+}(t)\cap B|\geq m/2$ holds as well. Setting $A'$ to be an $\lceil m/2\rceil$-sized subset of $H^{-}(t)\cap A$, and $B'$ to be an $\lceil m/2\rceil$-sized subset of $H^{+}(t)\cap B$, $A'$ and $B'$ are separated by $H(t)$.
\end{proof}

\begin{claim}\label{cl:partition}
Let $X,Y\subs V$ such that $|X|=|Y|=m$. There exist $X'\subs X$, $Y'\subs Y$, and $s\in [r]$ such that $|X'|=|Y'|\geq m2^{-r^{2}}$ and $X'\preceq_{s} Y'$ or $Y'\preceq_{s}X'$.
\end{claim}

\begin{proof}
Let $A,B\subs \mathbb{R}^{r}$ be the multisets defined as $A=\{\rank(x):x\in X\}$ and $B=\{\rank(y):y\in Y\}$. For $i,j\in [r]$, let $H_{i,j}(t)$ denote the hyperplane in $\mathbb{R}^{r}$ given by the equation $x_{i}-x_{j}=t$, and let $H'_{i,j}(t)$ be the hyperplane given by the equation $x_{i}+x_{j}=t$.

By repeatedly applying \Cref{cl:separation}, we can find $A'\subs A$, $B'\subs B$ and $t_{i,j},t_{i',j'}\in\mathbb{R}$ for $1\leq i<j\leq r$ such that $|A'|=|B'|\geq m2^{-2\binom{r}{2}}>m2^{-r^{2}}$, and the hyperplanes $H_{i,j}(t_{i,j})$ and $H'_{i,j}(t'_{i.j})$ separate $A'$ and $B'$ for $1\leq i<j\leq r$. We will show that $A'\preceq_{s} B'$ or $B'\preceq_{s} A'$ for some $s\in [r]$. 

Fix some $\alpha\in A'$ and $\beta\in B'$. We know that $\alpha\preceq_s \beta$ or $\beta\preceq_s \alpha$ for some $s\in[r]$. We may assume that $\alpha\preceq_s \beta$, i.e., $\beta_s-\alpha_s=\max_{i\in[r]} |\beta_i-\alpha_i|$. We claim that $A'\preceq_s B'$. 

Indeed, we know that $\beta_s-\alpha_s\ge |\beta_i-\alpha_i|$, i.e., $\beta_s-\alpha_s\ge \beta_i-\alpha_i$ and $\beta_s-\alpha_s\ge \alpha_i-\beta_i$ for every $i\in[r]$. Equivalently, $\beta_s-\beta_i\ge \alpha_s-\alpha_i$ and $\beta_s+\beta_i\ge \alpha_s+\alpha_i$ for every $i\in [r]$. However, $A'$ and $B'$ are separated by some $H_{s,i}(t_{s,i})$, so we must have $\beta'_s-\beta'_i\ge t_{s,i} \ge \alpha'_s-\alpha'_i$, and hence $\beta'_s-\alpha'_s\ge \beta'_i-\alpha'_i$ for every $\alpha'\in A'$ and $\beta'\in B'$. Similarly, $A'$ and $B'$ are separated by some $H'_{s,i}(t'_{s,i})$, so we must have $\beta'_s+\beta'_i\ge t'_{s,i} \ge \alpha'_s+\alpha'_i$, and therefore $\beta'_s-\alpha'_s\ge \alpha'_i-\beta'_i$ for every $\alpha'\in A'$ and $\beta'\in B'$.

This means that for every $\alpha'\in A'$, $\beta'\in B'$ and $i\in[r]$, we have $\beta'_s-\alpha'_s\ge |\beta_i-\alpha_i|$, or in other words, $\alpha'\preceq_s \beta'$. So indeed, $A'\preceq_s B'$, and setting $X'=\{x\in X:\rank(x)\in A'\}$ and $Y'=\{y\in Y:\rank(y)\in B'\}$ finishes the proof.
\end{proof}

Now we are prepared to prove that there are no large bicliques in the compement of $G$

\begin{lemma} \label{emptybipartite}
There is a constant $C$ depending only on $r$ such that $\overline{G}$ does not contain a biclique with more than $C\cdot a$ vertices.
\end{lemma}

\begin{proof}
Suppose that $\overline{G}$ contains a biclique with parts $X$ and $Y$ of size $m$. By \Cref{cl:partition}, there are subsets $X'\subs X$ and $Y'\subs Y$ of size $|X'|=|Y'|\ge m2^{-r^{2}}$ such that $X'\preceq_{s} Y'$ or $Y'\preceq_{s} X'$ for some $s\in[r]$. Without loss of generality, assume $X'\preceq_s  Y'$. 
   
For every $\alpha\in [b]^{r}$, let us define $X_{\alpha}=X'\cap V_{\alpha}$ and $Y_{\alpha}=Y'\cap V_{\alpha}$. Also, let $\calS$ be the set of pairs $(\alpha,\beta)\in [b]^r\times [b]^r$ such that $X_{\alpha},Y_{\beta}\ne \emptyset$. Our plan is to give an upper bound on
\begin{equation} \label{cl18:equ1}
  |X'||Y'| = \left(\sum_{\alpha\in [b]^{r}}|X_{\alpha}|\right)\left(\sum_{\beta\in[b]^{r}}|Y_{\beta}|\right)
  = \sum_{(\alpha,\beta)\in \calS} |X_\alpha||Y_\beta|.
\end{equation}

First we bound $|X_\alpha||Y_\beta|$ for a fixed pair $(\alpha,\beta)\in \calS$. Let $k=|\beta-\alpha|_\infty$, then by $\alpha \preceq_s \beta$, we know $k=\beta_s-\alpha_s\ge |\beta_i-\alpha_i|$ for every $i\in [r]$. We also know that no element of $X_{\alpha}$ is comparable to any element of $Y_{\beta}$ with respect to $<_{s}$, so there are no edges between the corresponding vertices $h(X_{\alpha})$ and $h(Y_{\beta})$ in the graph $H^k$. Therefore, by \Cref{expander}, we have $|X_{\alpha}||Y_{\beta}|\le a^{2}(1+\delta)^{-k}$.

Now let us count the number of pairs $(\alpha,\beta)\in \calS$ such that $|\beta-\alpha|_{\infty}=k$. Fix one such pair $(\alpha^*,\beta^*)$. As $X'\preceq_s Y'$, we have $\alpha^*\preceq_s \beta$ and $\alpha\preceq_s \beta^*$ for every $(\alpha,\beta)\in \calS$.
Hence, if $\alpha\preceq_s \beta$ and $|\beta-\alpha|_{\infty}=k$, then 
\[ \alpha^*_s-k \leq \beta_s-k = \alpha_s \le \beta^*_s =\alpha^*_s+k, \]
and hence $|\beta-\alpha^*|_\infty= \beta_s-\alpha^*_s \le 2k$ and $|\beta^*-\alpha|_\infty= \beta^*_s-\alpha_s \le 2k$. 

This shows that the number of such pairs $(\alpha,\beta)$ is at most $(4k+1)^{2r}$, as there are at most $4k+1$ possibilities for each of the $r$ coordinates of $\alpha$ and $\beta$. Plugging this in \eqref{cl18:equ1}, we get 
\[ |X'||Y'|\leq \sum_{k=0}^{b}a^{2}(1+\delta)^{-k}(4k+1)^{2r}. \]
As $\delta>0$, the sum $\sum_{k=0}^{\infty}(1+\delta)^{-k}(4k+1)^{2r}$ converges, so there is a constant $D$ depending only on $r$ such that $|X'||Y'|\leq Da^{2}$. As $|X'|=|Y'|$, this implies $|X'|=|Y'|\leq \sqrt{D}a$, which gives $|X|=|Y|\le Ca$ for $C=2^{r^{2}}\sqrt{D}$.
\end{proof}

\begin{proof}[Proof of \Cref{epsilon}]
   By \Cref{nedges}, the maximum degree of $G$ is at most $b^{r}3^{b}< (\log n)^r n^{\eps/2}$. If $n$ is sufficiently large, this implies that $G$ has at most $n^{1+\eps}$  edges. Also, by \Cref{emptybipartite}, the complement of $G$ does not contain a biclique of size $Ca\le cn/(\log n)^{r},$ where $c$ is a constant depending only on $\eps$ and $r$. This finishes the proof.
\end{proof}

Our construction can be adjusted slightly to achieve that the comparability graphs of $<_{1},\dots,<_{r}$ are also disjoint, by changing the partial orders $\prec_{1},\dots,\prec_{r}$.  Indeed, define the partial orders $\prec'_{1},\dots,\prec'_{r}$ on $[b]^{r}$ such that $\alpha\prec'_{s}\beta$, if $|\beta-\alpha|_{\infty}=\beta_{s}-\alpha_{s}$ and $|\beta_{i}-\alpha_{i}|<\beta_{s}-\alpha_{s}$ for $i<s$. Then, clearly, any two points of $[b]^{r}$ are comparable by exactly one of the partial orders $\prec_{1}',\dots,\prec_{r}'$. A similar argument shows that the graph $G$ defined with these partial orders also satisfies the conditions of \Cref{epsilon}, but the proof is a bit more technical.

\section{Concluding remarks}\label{sect:open}

As we mentioned in the introduction, \Cref{basicthm} provides the best known lower bound on the number of pairwise disjoint or pairwise intersecting sets among any $n$ convex sets or half lines in the plane ($n^{1/5}$ and $n^{1/3}$, respectively). A construction of Larman, Matou\v{s}ek, Pach and T\"or\H{o}csik \cite{LMPT94} gives an upper bound of $n^{\log 2/\log 5} \approx n^{0.431}$ for both problems, and in the case of convex sets, this was further improved to $n^{\log 8/\log 169 }\approx n^{0.405}$ by Kyn\v{c}l \cite{K12}.

There are several other geometric questions where comparability graphs are useful. For example, one can consider a family of translates of a fixed convex set $C$ in the plane, or a family of $x$-monotone curves grounded at some vertical line $L$. (An $x$-monotone curve is said to be grounded at $L$, if it intersects $L$ in exactly one point: its left endpoint.) Both of these are well-studied geometric settings (see, e.g., \cite{K04,G00}), and the disjointness graph of both families can be written as the union of two comparability graphs. So in both cases, we can find a disjoint or intersecting subfamily of size $n^{1/3}$.

Any improvement in \Cref{basicthm} for $r=2$ would improve all of these lower bounds. The main message of \Cref{mainthm} is that we cannot hope for much improvement without additional geometric observations. We should mention that this was done for the translates of a convex set, where the optimum was shown to be $\Theta(\sqrt{n})$ using new ideas \cite{P03,K04}. Still, it would be very interesting to see if the $(\frac{\log n}{\log\log n})^{2/3}$ error term could be removed in \Cref{mainthm}. We propose the following problem:

\begin{problem}
Is there a graph $G$ on $n$ vertices that is the union of two perfect (or comparability) graphs, and the largest homogeneous set of $G$ has size $O(n^{1/3})$?
\end{problem}

Unfortunately, our construction for \Cref{mainthm} does not seem to generalize to the union of $r>2$ comparability graphs. In fact, we believe that an $\eps>0$ might exist such that any union of 3 comparability graphs on $n$ vertices contains a homogeneous set of size $n^{1/4+\eps}$. Such a result would also imply that the disjointness graph of $n$ convex sets contains a homogeneous subset of size $n^{1/5+\eps/2}$.

As we cannot improve the lower bounds, it would make sense to try to turn our upper bound into a construction of sets in the plane. Unfortunately, certain geometric restrictions make it unlikely that this can be done. For example, Fox, Pach and T\'oth $\cite{FPT10}$ proved that the disjointness graph of convex sets contains a complete or empty bipartite graph with parts of size $\Omega(n)$, so our construction in \Cref{thm:bipartite} certainly cannot be adapted to that setting. 

Something similar can be done, however, for general $x$-monotone curves. Inspired by the present paper, Pach and Tomon \cite{PT20} very recently constructed a set of curves whose disjointness graph has large chromatic number, and therefore cannot be covered with fewer than 4 comparability graphs. It would be interesting to decide if the 4 comparability graphs are needed for segments, as well.

\medskip

It also makes sense to consider the off-diagonal version of our Ramsey problem, i.e., where we are looking for the largest clique in the union of two (or more) comparability graphs, assuming there is no independent set of size $t$. By \Cref{basicthm}, if $G$ is the union of two comparability graphs on $n$ vertices and $G$ does not contain an independent set of size $t$, then $G$ contains a clique of size $\sqrt{n/t}$. For $t=3$, this is actually weaker than the well-known result of Ajtai, Koml\'os and Szemer\'edi \cite{AKS80} that says that every graph whose complement is triangle-free contains a clique of size $\Omega(\sqrt{n\log n})$. However, we believe that for unions of comparability graphs, an even stronger bound should hold.

\begin{conjecture}
Let $G$ be a graph on $n$ vertices that is the union of two perfect (or comparability) graphs such that the complement of $G$ does not contain a triangle. Then $\omega(G)=n^{1-o(1)}$.
\end{conjecture}

Let us point out, though, that an $\Omega(n)$ lower bound cannot hold. Indeed, let $P$ be a point set in the unit square, no two on a vertical line. We define two partial orders $<_1$ and $<_2$ as follows. For points $p=(x,y)$ and $p'=(x',y')$ in $P$ with $x<x'$, let $p<_1 p'$ if $y>y'$, and let $p<_2 p'$ if $y\le y'$ and the axis-parallel rectangle with diameter $pp'$ contains another point of $P$. Then the union of the comparability graphs does not contain an independent set of size 3. On the other hand, by mimicking an argument of Chen, Pach, Szegedy and Tardos \cite{CPST09}, it is not hard to show that for a uniformly random point set $P$, the largest clique has size $O(n\frac{(\log\log n)^2}{\log n})$ with high probability.

This question is motivated by, and it is closely related to the conflict-free coloring problem for rectangles. Given a point set $P$ in the plane, the rectangle Delaunay graph $D_r(P)$ is defined by connecting two points $p,q\in P$ if the closed axis-parallel rectangle with diagonal $pq$ does not contain any other points of $P$. The conflict-free coloring problem asks for the chromatic number $\chi(D_r(P))$, which, as observed by Har-Peled and Smorodinsky \cite{HS05}, is essentially determined by the independence number of $D_r(P)$. The construction of Chen, Pach, Szegedy and Tardos in \cite{CPST09} provides an example where the independence number has size $o(n)$.
On the other hand, observing that $D_r(P)$ contains no clique of size 5, and that its complement is the union of two comparability graphs, we see that any lower bound in our problem for $t=5$ would automatically imply a lower bound on $\alpha(D_r(P))$.

\subsubsection*{Acknowledgments}

We would like to thank Nabil Mustafa, J\'anos Pach and G\'eza T\'oth for fruitful discussions, and Matthew Kwan for drawing our attention to the result of \cite{RS98}.

\end{document}